\documentclass{amsart}
\usepackage[a4paper, total={6in, 8in}]{geometry}

\usepackage{amsmath}
\usepackage{amsfonts}
\usepackage{amsthm}		
\usepackage{amssymb}
\usepackage{graphicx}
\usepackage{float}

\newtheorem{theorem}{Theorem}[section]
\newtheorem{corollary}{Corollary}[theorem]

\begin{document}
\title{On geometric constants for discrete Morrey spaces}
\author[A.~Adam]{Adam Adam}
\address{Analysis and Geometry Group, Faculty of Mathematics and Natural Sciences,
Bandung Institute of Technology, Bandung 40132, Indonesia}
\email{adam\_adam@students.itb.ac.id}

\author[H.~Gunawan]{Hendra Gunawan}
\address{Analysis and Geometry Group, Faculty of Mathematics and Natural Sciences,
Bandung Institute of Technology, Bandung 40132, Indonesia}
\email{hgunawan@math.itb.ac.id}



\maketitle

\begin{abstract}
In this note we prove that the $n$-th Von Neumann-Jordan constant and the $n$-th James
constant for discrete Morrey spaces $\ell^p_q$ where $1\le p<q<\infty$ are both equal to
$n$. This result tells us that the discrete Morrey spaces are not uniformly non-$\ell^1$,
and hence they are not uniformly $n$-convex.
\end{abstract}

\section{Introduction}
Let $n\ge 2$ be a non-negative integer and $(X,\|\cdot\|)$ be a Banach space.
The $n$-th {\em Von Neumann-Jordan constant} for $X$ \cite{KTH} is defined by
$$
C_{NJ}^{(n)}(X) := \sup \bigg\{\frac{\sum_\pm \|u_1 \pm u_2 \pm \dots \pm u_n \|_{X}^{2}}{2^{n-1}
\sum_{i=1}^{n} \|u_i\|_X} : u_i \neq 0, i = 1,2,\dots,n\bigg\}
$$
and the $n$-th {\em James constant} for $X$ \cite{MNPZ} is defined by
$$
C_{J}^{(n)}(X) := \sup \{\min \|u_1 \pm u_2 \pm \dots \pm u_n\| : u_i \in S_X, i = 1,2,\dots,n \}.
$$
Note that in the definition of $C_{NJ}^{(n)}(X)$, the sum $\sum_\pm$ is taken over all possible
combinations of $\pm$ signs. Similarly, in the definition of
$C_{J}^{(n)}(X)$, the minimum is taken over all possible
combinations of $\pm$ signs, while $S_X$ denotes the unit sphere in $X$, that is, $S_X=\{u \in X : \|u\| = 1\}$.

We say that $X$ is uniformly $n$-convex \cite{GHP} if for every $\varepsilon \in  (0,n]$ there exists
a $\delta \in (0,1)$ such that for every $u_1,u_2,\dots,u_n \in S_X$ with $\|u_1 \pm u_2 \pm \cdots \pm u_n\|
\geq \varepsilon$ for all combinations of $\pm$ signs except for $\|u_1 + u_2 + \cdots + u_n\|$,
we have
$$
    \|u_1 + u_2 + \cdots + u_n\| \leq n(1-\delta).
$$
Meanwhile, we say that $X$ is uniformly non-$\ell_{n}^{1}$
\cite{B,J,W} if there exists a $\delta \in (0,1)$ such that for every $u_1,u_2,\dots,u_n \in S_X$ we have
$$
\min \|u_1 \pm u_2 \pm \dots \pm u_n\| \leq n(1-\delta).
$$
Note that for $n=2$, uniformly non-$\ell_{n}^{1}$ spaces are known as uniformly nonsquare spaces, while for $n=3$
they are known as uniformly non-octahedral spaces.
One may verify that if $X$ is uniformly $n$-convex, then $X$ is uniformly non-$\ell_n^1$ \cite{GHP}.

Now a few remarks about the two constants, and their associations with the uniformly non-$\ell_n^1$ and
uniformly $n$-convex properties.
\begin{itemize}
    \item $1 \leq C_{NJ}^{(n)}(X) \leq n$ and
    $C_{NJ}^{(n)}(X) = 1$ if and only if $X$ is a Hilbert space \cite{KTH}.
    \item $1 \leq C_{J}^{(n)}(X) \leq n$. If $\dim(X) = \infty$, then $\sqrt{n} \leq C_{J}^{(n)}(X) \leq n$.
    Moreover, if $X$ is a Hilbert space, then $C_{J}^{(n)}(X) = \sqrt{n}$ \cite{MNPZ}.
    \item $X$ is uniformly non-$\ell_{n}^{1}$ if and only if $C_{NJ}^{(n)}(X) < n$ \cite{KTH}.
    \item $X$ is uniformly non-$\ell_{n}^{1}$ if and only if $C_{J}^{(n)}(X) < n$ \cite{MNPZ}.
\end{itemize}
The last two statements tell us that if $C_{NJ}^{(n)}(X)=n$ or $C_{J}^{(n)}(X)=n$, then $X$ is not uniformly non-$\ell_n^1$
and hence not uniformly $n$-convex.

In this note, we shall compute the two constants for discrete Morrey spaces. Let $\omega :=
\mathbb{N} \cup \{0\}$ and $m = (m_1,m_2,\dots,m_d) \in \mathbb{Z}^d$. Define
$$
S_{m,N} := \{k \in \mathbb{Z}^d : \|k-m\|_\infty \leq N\}
$$
where $N \in \omega$ and $\|m\|_\infty = \max\{|m_i|:1\leq i \leq d\}$. Denote by $|S_{m,N}|$ the
cardinality of $S_{m,N}$ for $m \in \mathbb{Z}^d$ and $N \in \omega$. Then we have $|S_{m,N}| = (2N+1)^d$.

Now let $1 \leq p \leq q < \infty$. Define $\ell_{q}^{p}=\ell_{q}^{p}(\mathbb{Z}^d)$ to be the
discrete Morrey space as introduced in \cite{GKS}, which consists of all sequences
$x : \mathbb{Z}^d \rightarrow \mathbb{R}$ with
$$
    \|x\|_{\ell_{q}^{p}} := \underset{m \in \mathbb{Z}^d, N \in \omega}{\sup} |S_{m,N}|^{\frac{1}{q} -
    \frac{1}{p}} \bigg( \sum_{k \in S_{m,N}} |x_k|^p \bigg)^{\frac{1}{p}} < \infty,
$$
where $x:=(x_k)$ with $k\in\mathbb{Z}^d$.
One may observe that discrete Morrey spaces are Banach spaces \cite{GKS}. Note, in particular, that
for $p=q$, we have $\ell^p_q=\ell^q$.

From \cite{GKSS} we already know that $C_{NJ}(\ell_{q}^{p})= C_{J}(\ell_{q}^{p}) = 2$, so that
$\ell_{q}^{p}$ is not uniformly nonsquare. In this note, we shall show that
$C_{NJ}^{(n)}(\ell_{q}^{p})= C_{J}^{(n)}(\ell_{q}^{p}) = n$, which leads us to the
conclusion that $\ell_{q}^{p}$ is not uniformly non-$\ell_{n}^{1}$, which is sharper than the
existing result. (If $X$ is not uniformly non-$\ell_n^1$, then $X$ is not uniformly non-$\ell_{n-1}^1$, provided that
$n\ge3$.)

\section{Main Results}

The value of the $n$-th Von Neumann-Jordan constant and the $n$-th James constant for discrete
Morrey spaces are stated in the following theorems. To understand the idea of the proof, we
first present the result for $n=3$.

\begin{theorem}
For $1 \leq p < q < \infty$, we have $C_{NJ}^{(3)}(\ell_{q}^{p}(\mathbb{Z}^d)) =
C_{J}^{(3)}(\ell_{q}^{p}(\mathbb{Z}^d)) = 3$.
\end{theorem}

\begin{proof} To prove the theorem, it suffices for us to find $x^{(1)},x^{(2)},x^{(3)} \in \ell_{q}^{p}$ such that
$$\frac{\sum_\pm \|x^{(1)} \pm x^{(2)} \pm x^{(3)} \|_{\ell_{q}^{p}}^{2}}{2^{2}\sum_{i=1}^{3} \|x^{(i)}\|_{\ell_{q}^{p}}} = 3$$
for the Von Neumann-Jordan constant, and
$$
\min \|x^{(1)} \pm x^{(2)} \pm x^{(3)}\|_{\ell_{q}^{p}} = 3
$$
for the James constant.

{\em Case 1: $d=1$}. Let $j \in \mathbb{Z}$ be a nonnegative, even integer such that $j > 4^{\frac{q}{q-p}}-1$,
or equivalently
$$
(j+1)^{\frac{1}{q}-\frac{1}{p}} < 4^{-\frac{1}{p}}.
$$
Construct $x^{(1)},x^{(2)},x^{(3)} \in \ell_{q}^{p}(\mathbb{Z})$ as follows:

\begin{itemize}
\item $x^{(1)} = (x_{k}^{(1)})_{k \in \mathbb{Z}}$ is defined by
$$
    x_{k}^{(1)} = \begin{cases}
      1, & k = 0,j,2j,3j, \\
      0, & \text{otherwise}; \\
   \end{cases}
$$

\item $x^{(2)} = (x_{k}^{(2)})_{k \in \mathbb{Z}}$ is defined by
$$
    x_{k}^{(2)} = \begin{cases}
      1, & k = 0,j, \\
      -1, & k = 2j,3j, \\
      0, & \text{otherwise}; \\
   \end{cases}
$$

\item $x^{(3)} = (x_{k}^{(3)})_{k \in \mathbb{Z}}$ is defined by
$$
    x_{k}^{(3)} = \begin{cases}
      1, & k = 0,2j, \\
      -1, & k = j,3j, \\
      0, & \text{otherwise}. \\
   \end{cases}
$$
\end{itemize}
The three sequences are in the unit sphere of $\ell^p_q(\mathbb{Z})$. Indeed, for the first sequence, we have
\begin{align*}
        \|x^{(1)}\|_{\ell_{q}^{p}} &= \underset{m \in \mathbb{Z}, N \in \omega}{\sup} |S_{m,N}|^{\frac{1}{q} -
        \frac{1}{p}} \bigg( \sum_{k \in S_{m,N}} |x_{k}^{(1)}|^p \bigg)^{\frac{1}{p}} \\
        &= \underset{m \in \mathbb{Z}\cap[0,3j], N \in \mathbb{Z}\cap[0,3j/2]}{\sup} |S_{m,N}|^{\frac{1}{q} -
        \frac{1}{p}} \bigg( \sum_{k \in S_{m,N}} |x_{k}^{(1)}|^p \bigg)^{\frac{1}{p}} \\
        &= \max \{1,(j+1)^{\frac{1}{q}-\frac{1}{p}}2^{\frac{1}{p}},(2j+1)^{\frac{1}{q}-
        \frac{1}{p}}3^{\frac{1}{p}},(3j+1)^{\frac{1}{q}-\frac{1}{p}}4^{\frac{1}{p}}\}.
\end{align*}
Since $(3j+1)^{\frac{1}{q}-\frac{1}{p}}<(2j+1)^{\frac{1}{q}-\frac{1}{p}}<(j+1)^{\frac{1}{q}-\frac{1}{p}}
< 4^{-\frac{1}{p}}$, we get $\|x^{(1)}\|_{\ell_{q}^{p}} = 1$.
Similarly, one may observe that $\|x^{(2)}\|_{\ell_{q}^{p}} = \|x^{(3)}\|_{\ell_{q}^{p}} = 1$.

Next, we observe that
$$
    x_k^{(1)}+x_k^{(2)}+x_k^{(3)} =
    \begin{cases}
      3, & k = 0,\\
      1, & k = j,2j, \\
      -1, & k = 3j, \\
      0, & \text{otherwise};
   \end{cases}
$$
$$
    x_k^{(1)}+x_k^{(2)}-x_k^{(3)} =
    \begin{cases}
      3, & k = j, \\
      1, & k = 0,3j, \\
      -1, & k = 2j, \\
      0, & \text{otherwise};
   \end{cases}
$$
$$
    x_k^{(1)}-x_k^{(2)}+x_k^{(3)} =
    \begin{cases}
      3, & k = 2j, \\
      1, & k = 0,3j, \\
      -1, & k = j, \\
      0, & \text{otherwise};
   \end{cases}
$$
$$
   x_k^{(1)}-x_k^{(2)}-x_k^{(3)} =
   \begin{cases}
      3, & k = 3j, \\
      1, & k = j,2j, \\
      -1, & k = 0, \\
      0, & \text{otherwise}.
   \end{cases}
$$
We first compute that
\[
       \|x^{(1)}+x^{(2)} + x^{(3)}\|_{\ell_{q}^{p}} = \max \{3,(j+1)^{\frac{1}{q}-\frac{1}{p}}(3^p+1)^{\frac{1}{p}},
        (2j+1)^{\frac{1}{q}-\frac{1}{p}}(3^p+2)^{\frac{1}{p}},
        (3j+1)^{\frac{1}{q}-\frac{1}{p}}(3^p+3)^{\frac{1}{p}}\}.
\]
Notice that
\begin{itemize}
        \item $(j+1)^{\frac{1}{q}-\frac{1}{p}}(3^p+1)^{\frac{1}{p}} <
        \Bigl(\frac{3^p + 1^p  }{4}\Bigr)^{\frac{1}{p}} < (3^p)^\frac{1}{p} = 3.$

        \item $
        (2j+1)^{\frac{1}{q}-\frac{1}{p}}(3^p+2)^{\frac{1}{p}} < (j+1)^{\frac{1}{q}-
        \frac{1}{p}}(3^p+2)^{\frac{1}{p}} < \Bigl(\frac{3^p + 2  }{4}\Bigr)^{\frac{1}{p}} < 3.$

        \item $
        (3j+1)^{\frac{1}{q}-\frac{1}{p}}(3^p+3)^{\frac{1}{p}} < (j+1)^{\frac{1}{q}-
        \frac{1}{p}}(3^p+3)^{\frac{1}{p}} < \Bigl(\frac{3^p + 3}{4}\Bigr)^{\frac{1}{p}} < 3.$
\end{itemize}
Hence, we obtain $\|x^{(1)}+x^{(2)} + x^{(3)}\|_{\ell_{q}^{p}} = 3.$

Similarly, we have
$$
\|x^{(1)} \pm x^{(2)} \pm x^{(3)}\|_{\ell_{q}^{p}} = \underset{m \in \mathbb{Z}\cap[0,3j],
N \in \mathbb{Z}\cap[0,3j/2]}{\sup} |S_{m,N}|^{\frac{1}{q} - \frac{1}{p}}
\bigg( \sum_{k \in S_{m,N}} |x^{(1)}_k\pm x^{(2)}_k\pm x^{(3)}_k|^p \bigg)^{\frac{1}{p}} = 3
$$
for every combination of $\pm$ signs.

Consequently, $\frac{\sum_\pm \|x^{(1)} \pm x^{(2)} \pm x^{(3)} \|_{\ell_{q}^{p}}^{2}}{2^{2}
\sum_{i=1}^{3} \|x^{(i)}\|_{\ell_{q}^{p}}} = 3$ and $\min \|x^{(1)} \pm x^{(2)} \pm x^{(3)}\|_{\ell_{q}^{p}} = 3$,
so we come to the conclusion that
$$
    C_{NJ}^{(3)}(\ell_{q}^{p}(\mathbb{Z})) = C_{J}^{(3)}(\ell_{q}^{p}(\mathbb{Z})) = 3.
$$

{\em Case 2: $d > 1$}. Let $j \in \mathbb{Z}$ be a nonnegative, even integer such that
$j > 4^{\frac{q}{d(q-p)}}-1$, which is equivalent to
$$
(j+1)^{d(\frac{1}{q}-\frac{1}{p})} < 4^{-\frac{1}{p}}.
$$
We then construct $x^{(1)},x^{(2)},x^{(3)} \in \ell_{q}^{p}(\mathbb{Z}^d)$ as follows:
\begin{itemize}
\item $x^{(1)} = (x_k^{(1)})_{k \in \mathbb{Z}^d}$ is defined by
    $$
    x_{k}^{(1)} =
    \begin{cases}
      1, & k = (0,0,\dots,0),(j,0,\dots,0),(2j,0,\dots,0),(3j,0,\dots,0), \\
      0, & \text{otherwise}; \\
   \end{cases}
    $$
\item $x^{(2)} = (x_{k}^{(2)})_{k \in \mathbb{Z}^d}$ is defined by
    $$
    x_{k}^{(2)} =
    \begin{cases}
      1, & k = (0,0,\dots,0),(j,0,\dots,0), \\
      -1, & k = (2j,0,\dots,0),(3j,0,\dots,0), \\
      0, & \text{otherwise}; \\
   \end{cases}
    $$
\item $x^{(3)} = (x_{k}^{(3)})_{k \in \mathbb{Z}^d}$ is defined by
    $$
    x_{k}^{(3)} =
    \begin{cases}
      1, & k = (0,0,\dots,0),(2j,0,\dots,0), \\
      -1, & k = (j,0,\dots,0),(3j,0,\dots,0), \\
      0, & \text{otherwise}. \\
   \end{cases}
    $$
\end{itemize}

As in the case where $d=1$, one may observe that
\begin{align*}
        \|x^{(1)}\|_{\ell_{q}^{p}} &= \underset{m \in \mathbb{Z}^d, N \in \omega}{\sup} |S_{m,N}|^{\frac{1}{q}
        - \frac{1}{p}} \bigg( \sum_{k \in S_{m,N}} |x_{k}^{(1)}|^p \bigg)^{\frac{1}{p}} \\
        &= \max \{1,(j+1)^{d(\frac{1}{q}-\frac{1}{p})}2^{\frac{1}{p}},(2j+1)^{d(\frac{1}{q}-
        \frac{1}{p})}3^{\frac{1}{p}},(3j+1)^{d(\frac{1}{q}-\frac{1}{p})}4^{\frac{1}{p}}\} \\
        &= 1.
\end{align*}
We also get $\|x^{(2)}\|_{\ell_{q}^{p}} = \|x^{(3)}\|_{\ell_{q}^{p}} = 1.$
Moreover, through similar observation as in the 1-dimensional case, we have
$$
\|x^{(1)} \pm x^{(2)} \pm x^{(3)}\|_{\ell_{q}^{p}} = 3
$$
for every possible combinations of $\pm$ signs. It thus follows that
$$
C_{J}^{(3)}(\ell_{q}^{p}(\mathbb{Z}^d)) = \sup \{\min \|x_1 \pm x_2 \pm x_3\|_{\ell_{q}^{p}} : x_1,x_2,x_3
\in S_{\ell_{q}^{p}}\} = 3
$$
and
$$
C_{NJ}^{(3)}(\ell_{q}^{p}(\mathbb{Z}^d)) = \sup \bigg\{\frac{\sum_\pm \|x_1 \pm x_2 \pm x_3 \|_{\ell_{q}^{p}}^{2}}{2^{2}
\sum_{i=1}^{3} \|x_i\|_{\ell_{q}^{p}}} : x_i \neq 0, i = 1,2,3\bigg\} = 3.
$$
\end{proof}

We now state the general result for $n\ge3$. (The proof is also valid for $n=2$, which amounts to the work of \cite{GKS}.)

\begin{theorem}
    For $1 \leq p < q < \infty$, we have $C_{NJ}^{(n)}(\ell_{q}^{p}(\mathbb{Z}^d)) = C_{J}^{(n)}(\ell_{q}^{p}(\mathbb{Z}^d)) = n$.
\end{theorem}

\begin{proof}
As for $n=3$, we shall consider the case where $d=1$ first, and then the case where $d>1$ later.

{\em Case 1: $d = 1$}. Let $j \in \mathbb{Z}$ be a nonnegative, even integer such that
$j > 2^{(n-1)(\frac{q}{q-p})}-1$, which is equivalent to
$$
(j+1)^{\frac{1}{q}-\frac{1}{p}} < 2^{-\frac{(n-1)}{p}}.
$$
We construct $x^{(i)} \in \ell_{q}^{p} \in \mathbb{Z}$ for $i = 1,2,\dots,n$ as follows:

\begin{itemize}
\item  $x^{(1)} = (x_{k}^{(1)})_{k \in \mathbb{Z}}$ is defined by
$$
    x_{k}^{(1)} =
    \begin{cases}
      1, & k \in S_{1}^{(1)}, \\
      0, & \text{otherwise}, \\
   \end{cases}
$$
where
    \begin{align*}
        S_{1}^{(1)} = \{0,j,2j,3j,\dots, (2^{n-1}-1)j\};
    \end{align*}
\item $x^{(i)} = (x_{k}^{(i)})_{k \in \mathbb{Z}}$ for $2\leq i \leq n$ is defined by
$$
    x_{k}^{(i)} =
    \begin{cases}
      1, & k \in S_{1}^{(i)}, \\
      -1, & k \in S_{-1}^{(i)}, \\
      0, & \text{otherwise}, \\
   \end{cases}
$$
with the following rules: Write $P = \{0,j,2j,\dots, (2^{n-1}-1)j\}$ as
$$
    P = P_{1}^{(i)} \cup P_{2}^{(i)} \cup \dots \cup P_{2^{i-1}}^{(i)}
$$
where $P_{1}^{(i)}$ consists of the first $\frac{2^{n-1}}{2^{i-1}}$ terms of $P$, $P_{2}^{(i)}$ consists of the next
$\frac{2^{n-1}}{2^{i-1}}$ terms of $P$, and so on. Then $S_1^{(i)}$ and $S_{-1}^{(i)}$ are given by
    \begin{align*}
    S_{1}^{(i)} &= P_{1}^{(i)} \cup P_{3}^{(i)} \cup \dots \cup P_{2^{i-1}-1}^{(i)}, \\
    S_{-1}^{(i)} &= P_{2}^{(i)} \cup P_{4}^{(i)} \cup \dots \cup P_{2^{i-1}}^{(i)}.
    \end{align*}
For example, for $i=2$, $x^{(2)} = (x_{k}^{(2)})_{k \in \mathbb{Z}}$ is defined by
$$
    x_{k}^{(2)} =
    \begin{cases}
      1, & k \in S_{1}^{(2)}, \\
      -1, & k \in S_{-1}^{(2)}, \\
      0, & \text{otherwise}, \\
   \end{cases}
$$
where
    \begin{align*}
        S_{1}^{(2)} &= \bigg\{0,j,2j,3j,\dots, \Bigl(\frac{2^{n-1}}{2}-1\Bigr)j\bigg\} \\
        S_{-1}^{(2)} &= \bigg \{\Bigl(\frac{2^{n-1}}{2}\Bigr)j,\Bigl(\frac{2^{n-1}}{2}+1\Bigr)j,\dots, (2^{n-1}-1)j\bigg\};
    \end{align*}
\end{itemize}

Note that the largest absolute value of the terms of $x^{(i)}$ in the above construction will be equal to $1$ for each $i=1,\dots,n$.
Next, since the number of possible combinations of $\pm$ signs in $x^{(1)} \pm x^{(2)} \pm \dots \pm x^{(n)}$
is $2^{n-1}$, the above construction will give us $1+1+\dots+1 = n$ as the largest absolute value of
$x^{(1)} \pm x^{(2)} \pm \dots \pm x^{(n)}$ for every combination of $\pm$ signs. This means that, if $x^{(1)} \pm x^{(2)} \pm \dots
\pm x^{(n)} = (x_k)_{k \in \mathbb{Z}}$, then $\max\limits_{k\in\mathbb{Z}} |x_k| = n$.

Let us now compute the norms. For $x^{(1)}$, we have
    \begin{align*}
        \|x^{(1)}\|_{\ell_{q}^{p}} &= \underset{m \in \mathbb{Z}, N \in \omega}{\sup} |S_{m,N}|^{\frac{1}{q} - \frac{1}{p}} \bigg( \sum_{k \in S_{m,N}} |x_{k}^{(1)}|^p \bigg)^{\frac{1}{p}} \\
        &= \underset{m \in \mathbb{Z}\cap[0,(2^{n-1}-1)j], N \in \mathbb{Z}\cap[0,(2^{n-1}-1)j/2]}{\sup} |S_{m,N}|^{\frac{1}{q} - \frac{1}{p}} \bigg( \sum_{k \in S_{m,N}} |x_{k}^{(1)}|^p \bigg)^{\frac{1}{p}}
        \\
        &= \max \{1,(j+1)^{\frac{1}{q}-\frac{1}{p}}2^{\frac{1}{p}}, (2j+1)^{\frac{1}{q}-\frac{1}{p}}3^{\frac{1}{p}},\dots,
        ((2^{n-1}-1)j+1)^{\frac{1}{q}-\frac{1}{p}}2^{\frac{n-1}{p}}\}.
    \end{align*}
For each $r=1,2,\dots,2^{n-1}-1$, we have $(rj+1)^{\frac{1}{q}-\frac{1}{p}} \leq (j+1)^{\frac{1}{q}-\frac{1}{p}}$ and
$(r+1)^{\frac{1}{p}} \leq 2^{\frac{n-1}{p}}$, so that
$$
    (rj+1)^{\frac{1}{q}-\frac{1}{p}}(r+1)^{\frac{1}{p}} \leq (j+1)^{\frac{1}{q}-\frac{1}{p}}2^{\frac{n-1}{p}} <
    2^{-\frac{n-1}{p}}2^{\frac{n-1}{p}} = 1.
$$
Hence we obtain $\|x^{(1)}\|_{\ell_{q}^{p}} = 1$. Similarly, one may verify that
$$
\|x^{(2)}\|_{\ell_{q}^{p}} = \|x^{(3)}\|_{\ell_{q}^{p}} = \dots = \|x^{(n)}\|_{\ell_{q}^{p}} = 1.
$$
Next, we shall compute the norms of
$
    x^{(1)} \pm x^{(2)} \pm \dots \pm x^{(n)}.
$
Write $x^{(1)} + x^{(2)} + \dots + x^{(n)} = (x_k)_{k \in \mathbb{Z}}$ where
$$
    x_k :=
    \begin{cases}
      a_1, & k = 0,  \\
      a_2, & k = j, \\
      a_3, &k = 2j, \\
      \ \vdots & \\
      a_{2^{n-1}}, &k = (2^{n-1}-1)j, \\
      0, & \text{otherwise}, \\
   \end{cases}
$$
with $a_1=n$ and $|a_i| < n$ for $i=2,3,\dots, (2^{n-1})j$.
Accordingly, we have
    \begin{align*}
        \|x^{(1)}+x^{(2)} + \dots + x^{(n)}\|_{\ell_{q}^{p}} = &\underset{m \in \mathbb{Z}, N \in \omega}
        {\sup} |S_{m,N}|^{\frac{1}{q} - \frac{1}{p}} \bigg( \sum_{k \in S_{m,N}} |x_k|^p \bigg)^{\frac{1}{p}} \\
        = &\underset{m \in \mathbb{Z}\cap[0,(2^{n-1}-1)j], N \in \mathbb{Z}\cap[0,(2^{n-1}-1)j/2]}
        {\sup} |S_{m,N}|^{\frac{1}{q} - \frac{1}{p}} \bigg( \sum_{k \in S_{m,N}} |x_k|^p \bigg)^{\frac{1}{p}}\\
        = &\max \Bigl\{n,(j+1)^{\frac{1}{q}-\frac{1}{p}}(n^p+a_2^p)^{\frac{1}{p}},
        (2j+1)^{\frac{1}{q}-\frac{1}{p}}(n^p+a_2^p+a_3^p)^{\frac{1}{p}},\\
        &\dots,((2^{n-1}-1)j+1)^{\frac{1}{q}-\frac{1}{p}} \bigl(n^p + \sum_{i=2}^{2^{n-1}}a_{i}^{p}\bigr)^{\frac{1}{p}}\Bigr\}.
    \end{align*}
Since $(rj+1)^{\frac{1}{q}-\frac{1}{p}} \leq (j+1)^{\frac{1}{q}-\frac{1}{p}}$ for each $r=1,2,\dots,2^{n-1}-1$, we obtain
    \begin{align*}
    (rj+1)^{\frac{1}{q}-\frac{1}{p}}\bigl(n^p + \sum_{i=2}^{r+1}a_{i}^{p}\bigr)^{\frac{1}{p}} &\leq
    (j+1)^{\frac{1}{q}-\frac{1}{p}}\bigl(n^p + \sum_{i=2}^{r+1}a_{i}^{p}\bigr)^{\frac{1}{p}}    \\
    &< 2^{-\frac{(n-1)}{p}}\bigl(n^p + \sum_{i=2}^{r+1}a_{i}^{p}\bigr)^{\frac{1}{p}} \\
    &< 2^{-\frac{(n-1)}{p}}(\underbrace{n^p + n^p + \dots+ n^p}_\text{$r+1$ times} )^{\frac{1}{p}} \\
    &= 2^{-\frac{(n-1)}{p}} (r+1)^{\frac{1}{p}}(n^p)^{\frac{1}{p}} \\
    &\leq 2^{-\frac{(n-1)}{p}}2^{\frac{(n-1)}{p}}n \\
    &= n.
    \end{align*}
It thus follows that
$$
    \|x^{(1)}+x^{(2)} + \dots + x^{(n)}\|_{\ell_{q}^{p}} = n.
$$

As we have remarked earlier, the largest absolute value of $x^{(1)} \pm x^{(2)} \pm \dots \pm x^{(n)}$ is equal to $n$
for every combination of $\pm$ signs. 
Moreover, it is clear that for $k \notin \{0,2j,\dots,(2^{n-1}-1)j\}$, the $k$-th term of $x^{(1)} \pm x^{(2)} \pm \dots \pm x^{(n)}$ is equal to 0.
Hence, we obtain
    \begin{align*}
        \|x^{(1)} \pm x^{(2)} \pm \dots \pm x^{(n)}\|_{\ell_{q}^{p}} = \underset{m \in \mathbb{Z}, N \in \omega} {\sup} &|S_{m,N}|^{\frac{1}{q} -
        \frac{1}{p}} \bigg( \sum_{k \in S_{m,N}} |x^{(1)}_k \pm x^{(2)}_k \pm \dots \pm x^{(n)}_k|^p \bigg)^{\frac{1}{p}} \\
        = \underset{m \in \mathbb{Z}\cap[0,(2^{n-1}-1)j], N \in \mathbb{Z}\cap[0,(2^{n-1}-1)j/2]} {\sup} &|S_{m,N}|^{\frac{1}{q} -
        \frac{1}{p}} \bigg( \sum_{k \in S_{m,N}} |x^{(1)}_k \pm x^{(2)}_k \pm \dots \pm x^{(n)}_k|^p \bigg)^{\frac{1}{p}}=n.
    \end{align*}
Consequently, we get
$$
    \frac{\sum_{\pm} \|x^{(1)} \pm x^{(2)} \pm \dots \pm x^{(n)}\|_{\ell_{q}^{p}}^{2}}{2^{n-1}\sum_{i=1}^{n}
    \|x_i\|_{\ell_{q}^{p}}} = \frac{2^{n-1}n^2}{2^{n-1}n} = n
$$
and
$$
    \min \|x^{(1)} \pm x^{(2)} \pm \dots \pm x^{(n)}\|_{\ell_{q}^{p}} = n,
$$
whence
$$
    C_{NJ}^{(n)}(\ell_{q}^{p}(\mathbb{Z})) =  C_{J}^{(n)}(\ell_{q}^{p}(\mathbb{Z})) = n.
$$

{\em Case 2: $d > 1$}. Here we choose $j \in \mathbb{Z}$ to be a nonnegative, even integer such that
$j > 2^{(\frac{n-1}{d})(\frac{q}{q-p})}-1$ or, equivalently,
$$
(j+1)^{d(\frac{1}{q}-\frac{1}{p})} < 2^{-\frac{(n-1)}{p}}.
$$
Then, using the sequences
$$
            x^{(i)} = (x^{(i)}_{k_1})_{k_1 \in \mathbb{Z}} \in \ell_{q}^{p}(\mathbb{Z}),\quad i=1,\dots,n,
$$
in the case where $d=1$, we now define $x^{(i)} := (x^{(i)}_{k})_{k \in \mathbb{Z}^d} \in \ell_{q}^{p}(\mathbb{Z}^d)$ for
$i=1,\dots,n$, where
$$
            x^{(i)}_k =
            \begin{cases}
            x^{(i)}_{k_1}, &k = (k_1,0,0,\dots,0), \\
            0, & \text{otherwise}.
            \end{cases}
$$
We shall then obtain
$$
    C_{NJ}^{(n)}(\ell_{q}^{p}(\mathbb{Z}^d)) = C_{J}^{(n)}(\ell_{q}^{p}(\mathbb{Z}^d)) = n,
$$
as desired.
\end{proof}

\begin{corollary}
    For $1 \leq p < q < \infty$, the space $\ell_{q}^{p}$ is not uniformly non-$\ell_{n}^{1}$.
\end{corollary}

\begin{corollary}
    For $1 \leq p < q < \infty$, the space $\ell_{q}^{p}$ is not uniformly $n$-convex.
\end{corollary}

\bigskip

\noindent{\bf Acknowledgement}. The work is part of the first author's thesis. Both authors are supported by P2MI 2021
Program of Bandung Institute of Technology.

\end{document}